\documentclass[12pt]{amsart}

\usepackage[a4paper,margin=28mm]{geometry}
\usepackage{amsmath,amssymb,amsthm,mathtools,esint}
\usepackage{microtype}
\usepackage[colorlinks=true,linkcolor=blue,citecolor=blue,urlcolor=blue]{hyperref}

\numberwithin{equation}{section}
\newtheorem{theorem}{Theorem}[section]
\newtheorem{proposition}[theorem]{Proposition}
\newtheorem{lemma}[theorem]{Lemma}
\newtheorem{corollary}[theorem]{Corollary}
\theoremstyle{remark}
\newtheorem{remark}[theorem]{Remark}

\DeclareMathOperator{\Ric}{Ric}
\DeclareMathOperator{\divg}{div}
\DeclareMathOperator{\Vol}{Vol}
\DeclareMathOperator{\osc}{osc}
\newcommand{\dd}{\,dV_g}
\newcommand{\R}{\mathbb R}
\newcommand{\ip}[2]{\langle #1,#2\rangle}

\title[Exponentially harmonic functions]
{Gradient estimates and Liouville theorems\\
for exponentially harmonic functions}
\author{Xishen Jin }
	\address{Xishen Jin\\ Department of Mathematics\\ Renmin University of China \\ Beijing\\ 100872\\ China} \email{jinxishen@ruc.edu.cn}
\author{Xi Zhang}
\address{School of Mathematics and Statistics,
Nanjing University of Science and Technology,
Nanjing 210094, P.R. China}
\email{mathzx@njust.edu.cn}
\date{}

\begin{document}

\begin{abstract}
We study exponentially harmonic functions on complete Riemannian
manifolds. Under a lower Ricci curvature bound \(\Ric_g\geq-Kg\), we first establish a local
gradient estimate in terms of the oscillation. We then prove a
pointwise gradient estimate for solutions bounded from one side. As an application, when \(\Ric_g\geq0\), we obtain a Liouville
theorem under a one-sided sublinear growth condition. In particular,
every exponentially harmonic function bounded above or below is
constant.
\end{abstract}

\maketitle
\enlargethispage{2pt}

\section{Introduction}

Let \((M^n,g)\) be a complete connected Riemannian manifold without
boundary, where \(n\geq2\). For any relatively compact domain
\(\Omega\subset M\), we define the exponential energy
\begin{equation}\label{eq:energy-functional}
E_\Omega(u)=\int_\Omega\left(e^{|\nabla u|^2/2}-1\right)\dd.
\end{equation}
A real-valued function \(u\) is called exponentially harmonic if it
satisfies the Euler--Lagrange equation associated with $E_{\Omega}$
\begin{equation}\label{eq:eh}
\divg\!\left(e^{|\nabla u|^2/2}\nabla u\right)=0.
\end{equation}
For \(u\in C^2(\Omega)\), this is equivalent to a quasilinear equation
\begin{equation}\label{eq:nondiv}
\Delta u+\nabla^2u(\nabla u,\nabla u)=0.
\end{equation}

Gradient estimates are fundamental in the theory of harmonic functions. Yau \cite{Yau} proved that every positive harmonic function on a complete manifold with nonnegative Ricci curvature is constant. The local gradient estimate of Cheng-Yau \cite{CY}
gives a quantitative form of this result. If \(v>0\) is harmonic on
\(B_{2R}(x_0)\) and \(\Ric_g\geq-Kg\), then
\begin{equation}\label{eq:harmonic-background}
\sup_{B_R(x_0)}\frac{|\nabla v|}{v}
\leq C_n\left(\frac1R+\sqrt K\right).
\end{equation}
Integration along curves yields a local Harnack inequality. Applying
\eqref{eq:harmonic-background} to \(u-c+\varepsilon\) also gives a
gradient estimate when \(u\geq c\). It follows that a global harmonic function bounded from one
side is constant under nonnegative Ricci curvature. The same
estimate also implies constancy under one-sided sublinear growth.

Corresponding results also hold for \(p\)-harmonic functions, which are weak solutions of
\[
\Delta_p v:=\divg\!\left(|\nabla v|^{p-2}\nabla v\right)=0,
\qquad 1<p<\infty.
\]
Rigoli-Salvatori-Vignati \cite{RSV} proved Liouville-type results
for \(p\)-harmonic functions under volume doubling and Poincar\'e
inequalities. Their results include the constancy of one-sided bounded functions on complete manifolds with nonnegative Ricci curvature. Kotschwar-Ni \cite{KN} obtained local gradient
estimates under a lower bound on sectional curvature. Wang-Zhang \cite{WZ} proved the estimate under a Ricci lower bound.
For a positive \(p\)-harmonic function on \(B_{2R}(x_0)\), their
result gives
\begin{equation}\label{eq:p-harmonic-background}
\sup_{B_R(x_0)}\frac{|\nabla v|}{v}
\leq C_{n,p}\left(\frac1R+\sqrt K\right).
\end{equation}
The key tools in their proof are Moser iteration and a local Sobolev inequality.
As in the harmonic case, this implies a Harnack inequality and
Liouville theorems under one-sided bounds or one-sided sublinear
growth when \(K=0\). Sung-Wang \cite{SW} obtained sharp gradient
estimates for positive eigenfunctions of the \(p\)-Laplacian under
a lower Ricci bound and also studied the associated rigidity
problem.

Exponentially harmonic maps were first studied by Eells-Lemaire
\cite{EL}. Duc-Eells \cite{DE} considered regularity questions
for exponentially harmonic functions. Lieberman \cite{Lieberman}
proved regularity for minimizers of the exponential energy.
More recently, Marcellini-Nastasi-Pacchiano Camacho \cite{MNP}
obtained local Lipschitz estimates for minimizers under general
growth assumptions, covering both \(p,q\)-growth and
exponential growth.

Regarding Liouville theorems, Hong \cite{Hong} proved that bounded
exponentially harmonic functions are constant on complete manifolds
with nonnegative sectional curvature. Wu-Ruan-Yang \cite{WRY}
proved a gradient estimate via the maximum principle and obtained
a Liouville theorem under nonnegative Ricci curvature together with
a lower bound on the sectional curvature.
In Euclidean domains, Moreira-Santos-Monari Soares \cite{MSS24}
proved a quantitative Hopf--Oleinik lemma and an interior gradient
estimate for \(\Delta v+2\nabla^2v(\nabla v,\nabla v)=0\), which agrees with \eqref{eq:nondiv} after setting
\(u=\sqrt2\,v\).
In \cite{MSS25}, they proved a classification theorem in half-spaces and showed that entire solutions on \(\R^n\) with at most
linear growth are affine. Their Liouville corollary
assumes a linear growth bound.

In this paper, we prove local gradient estimates under the sole
curvature assumption \(\Ric_g\geq-Kg\). We first estimate the gradient in terms of the oscillation. For a bounded function on a set \(U\), write
\[
\osc_Uu=\sup_Uu-\inf_Uu.
\]
Our first result is the following.

\begin{theorem}\label{thm:gradient}
Let \((M^n,g)\) be a complete connected Riemannian manifold without boundary with
\(
\Ric_g\geq-Kg\) for some \(K\geq0.
\)
Let \(p\in M\) and \(R>0\). If
\(u\in C^3(B_{4R}(p))\cap L^\infty(B_{4R}(p))\)
is exponentially harmonic, then
\begin{equation}\label{eq:gradient}
\sup_{B_R(p)}|\nabla u|
\leq C_n\left(\frac1R+\sqrt K\right)
        \osc_{B_{4R}(p)}u,
\end{equation}
where \(C_n\) depends only on \(n\).
\end{theorem}

\begin{remark}
Throughout the paper, we 
state our results for $C^3$ solutions to justify the
differential calculations below.
The results also hold for local minimizers
$u\in W^{1,1}_{\mathrm{loc}}$ with locally finite
exponential energy.
Indeed, such minimizers are smooth by
\cite[Section~2]{Lieberman}.

In particular, the results hold for weak solutions in
$W^{1,\infty}_{\mathrm{loc}}$.
The convexity of the energy implies that each such
solution is a local minimizer.
\end{remark}

We next prove a gradient estimate for functions bounded
from one side. For this purpose, define
\begin{equation}\label{eq:psi-definition}
k=\max\{n-1,2\},\qquad
\Psi_k(t)=t^{1/k}\exp\!\left(\frac{t^2}{2k}\right),
\qquad t\geq0.
\end{equation}
Note that \(\Psi_k\) is strictly increasing and \(\Psi_k(0)=0\).

\begin{theorem}\label{thm:one-sided}
Let \((M^n,g)\) be a complete connected Riemannian manifold without boundary with
\(
\Ric_g\geq-Kg\) for some \(K\geq0.
\)
Let \(u\in C^3(B_{4R}(p))\) be exponentially harmonic.
If \(u\geq c\) on \(B_{4R}(p)\) for some \(c\in\R\), then
\begin{equation}\label{eq:one-sided-gradient}
|\nabla u(x)|
\leq C_n\Psi_k\!\left[
\left(\frac1R+\sqrt K\right)(u(x)-c)\right],
\qquad x\in B_R(p).
\end{equation}
\end{theorem}

\begin{remark}
    If \(u\leq c\) on \(B_{4R}(p)\), the same estimate \eqref{eq:one-sided-gradient} holds with
\(u(x)-c\) replaced by \(c-u(x)\).
\end{remark}

When \(K=0\), we obtain the following Liouville theorem
under a one-sided sublinear growth condition.

\begin{theorem}\label{thm:liouville}
Let \((M^n,g)\) be a complete connected Riemannian manifold
without boundary with \(\Ric_g\geq0\).
Let \(u\in C^3(M)\) be exponentially harmonic.
If, for some \(p\in M\), either
\begin{equation}\label{eq:one-sided-growth}
u(p)-\inf_{B_R(p)}u=o(R)
\quad\text{or}\quad
\sup_{B_R(p)}u-u(p)=o(R)
\qquad\text{as }R\to\infty,
\end{equation}
then \(u\) is constant.
\end{theorem}

As a consequence, we get the following Liouville theorem for one-sided bounded exponentially harmonic functions.

\begin{corollary}\label{cor:bounded}
Every \(C^3\) exponentially harmonic function bounded above or below
on a complete connected Riemannian manifold with nonnegative Ricci
curvature is constant. In particular, every bounded exponentially
harmonic function on such a manifold is constant.
\end{corollary}

The paper is organized as follows.
Section~\ref{sec:energy} contains the energy comparison and the
Bochner identity. As an application of the Bochner identity, we also prove that if \(\Ric_g\geq0\) and \(|\nabla u|\) attains
its global maximum, then either \(u\) is constant or \(M\) splits
isometrically as \(\R\times N\), with \(u\) affine in the
\(\R\)-direction; see Corollary~\ref{cor:gradient-rigidity}. In Section~\ref{sec:iteration}, we establish the mean-value
estimate. We prove the oscillation gradient estimate in
Section~\ref{sec:gradient} and the one-sided gradient estimate in
Section~\ref{sec:one-sided}. Section~\ref{sec:liouville} contains
the Liouville theorems.

\section{Energy comparison and a Bochner identity}\label{sec:energy}

Throughout the paper, \(C_n\) denotes a positive constant depending
only on \(n\) whose value may change from line to line.
All norms, derivatives, and geodesic balls are taken with respect to the metric \(g\).
For a measurable set \(U\) of positive finite volume, write
\[
\fint_Uf\dd=\frac1{\Vol(U)}\int_Uf\dd.
\]

\subsection{Energy Comparison}

We first record the volume estimate used in the proof.
Suppose that \(\Ric_g\geq-Kg\) and \(K\rho^2\leq1\).
By Bishop--Gromov volume comparison theorem,
\begin{equation}\label{eq:doubling}
\frac{\Vol(B_{4\rho}(x))}{\Vol(B_{2\rho}(x))}\leq 2^n\exp(C_n\sqrt K\,\rho)
\leq C_n.
\end{equation}

For each \(x\in M\), define
\[
\Phi(\xi)=e^{|\xi|^2/2}-1,\qquad \xi\in T_xM.
\]
Its derivatives satisfy
\begin{align}
D\Phi(\xi)[\zeta]
 &=e^{|\xi|^2/2}\ip{\xi}{\zeta},\label{eq:convex-first}\\
D^2\Phi(\xi)[\zeta,\zeta]
 &=e^{|\xi|^2/2}
   \left(|\zeta|^2+\ip{\xi}{\zeta}^2\right).
\label{eq:convex-second}
\end{align}
In particular, \(\Phi\) is strictly convex in the gradient variable.

\begin{lemma}\label{lem:minimality}
Let \(\Omega \subset M\) be open and \(u \in C^3(\Omega)\) solve \eqref{eq:eh}.
Let \(v\) be locally Lipschitz on \(\Omega\) and assume that
\(v-u\) has compact support in \(\Omega\).
Then for every open set \(\Omega'\Subset\Omega\) containing this support,
\begin{equation}\label{eq:minimality}
\int_{\Omega'}\Phi(\nabla u)\dd
\leq\int_{\Omega'}\Phi(\nabla v)\dd.
\end{equation}
\end{lemma}

\begin{proof}
By convexity of $\Phi$, we have
\[
\Phi(\nabla v)-\Phi(\nabla u)
\geq e^{|\nabla u|^2/2}
       \ip{\nabla u}{\nabla(v-u)}.
\]
The integral of the right-hand side vanishes by \eqref{eq:eh}.
Hence, \eqref{eq:minimality} follows.
\end{proof}

\begin{lemma}\label{lem:energy}
Let \(u\in C^3(B_{4\rho}(x))\cap L^\infty(B_{4\rho}(x))\)
solve \eqref{eq:eh}. Set
\[
S=\osc_{B_{4\rho}(x)}u,\qquad
W=e^{|\nabla u|^2/2}-1.
\]
Then
\begin{equation}\label{eq:energy-volume}
\fint_{B_{2\rho}(x)}W\dd
\leq
2\frac{\Vol(B_{4\rho}(x))}{\Vol(B_{2\rho}(x))}
\left(e^{S^2/(32\rho^2)}-1\right).
\end{equation}
If \(\Ric_g\geq-Kg\) and \(K\rho^2\leq1\), then
\begin{equation}\label{eq:average-energy}
\fint_{B_{2\rho}(x)}W\dd
\leq a_n\left(e^{S^2/(32\rho^2)}-1\right)
\end{equation}
for some constant \(a_n\geq1\) depending only on \(n\).
\end{lemma}

\begin{proof}
For the energy comparison, we construct two functions
using the distance function. Choose
\[
a<\inf_{B_{4\rho}(x)}u,\qquad
b>\sup_{B_{4\rho}(x)}u.
\]
Write \(S_{a,b}=b-a\) and \(r(y)=d_g(x,y)\).
Define
\[
q_+(y)=a+\frac{S_{a,b}}{4\rho}r(y),\qquad
q_-(y)=b-\frac{S_{a,b}}{4\rho}r(y),
\]
and
\[
v_+=\min\{u,q_+\},\qquad v_-=\max\{u,q_-\}.
\]
Let
\begin{align*}
D_+&=\{y\in B_{4\rho}(x):u(y)>q_+(y)\},\\
D_-&=\{y\in B_{4\rho}(x):u(y)<q_-(y)\}.
\end{align*}
The strict choices of \(a\) and \(b\) imply that both
\(v_+-u\) and \(v_--u\) have compact support in \(B_{4\rho}(x)\).

Since \(|\nabla r|\leq1\) almost everywhere, we have
\[
|\nabla q_\pm|\leq\frac{S_{a,b}}{4\rho}
\quad\text{almost everywhere}.
\]
The gradients of \(v_+\) and \(u\) agree almost everywhere outside
\(D_+\). On \(D_+\), the gradient of \(v_+\) coincides with that of \(q_+\).
Applying Lemma~\ref{lem:minimality} to \(v_+\), we obtain
\[
\int_{D_+}W\dd
\leq \Vol(D_+)
   \left(e^{S_{a,b}^2/(32\rho^2)}-1\right).
\]
The same argument for \(v_-\) gives
\[
\int_{D_-}W\dd
\leq \Vol(D_-)
   \left(e^{S_{a,b}^2/(32\rho^2)}-1\right).
\]

On \(B_{2\rho}(x)\), we have \(q_+<q_-\).
Every point of this ball therefore belongs to \(D_+\cup D_-\).
Since \(W\geq0\), it follows that
\[
\int_{B_{2\rho}(x)}W\dd
\leq2\Vol(B_{4\rho}(x))
   \left(e^{S_{a,b}^2/(32\rho^2)}-1\right).
\]
Dividing \(\Vol(B_{2\rho}(x))\) and letting
\(a\uparrow\inf u\) and \(b\downarrow\sup u\) yields \eqref{eq:energy-volume}.
Inequality \eqref{eq:average-energy} follows from
\eqref{eq:doubling}.
\end{proof}

\subsection{A Bochner identity}\label{sec:bochner}

For a solution of \eqref{eq:eh}, set
\begin{equation}\label{eq:hAW}
h=\frac12|\nabla u|^2,\qquad W=e^h-1,\qquad
A(X)=X+\ip{\nabla u}{X}\nabla u.
\end{equation}
The endomorphism \(A\) is symmetric and positive definite.

\begin{lemma}\label{lem:bochner}
If \(u\in C^3(\Omega)\) solves \eqref{eq:eh}, then
\begin{equation}\label{eq:bochner}
\divg(A\nabla W)
=e^h\left(
 |\nabla^2u|^2+\Ric(\nabla u,\nabla u)+|\nabla h|^2
 \right).
\end{equation}
Moreover,
\begin{equation}\label{eq:ellipticity}
|X|^2\leq\ip{AX}{X}\leq\lambda|X|^2,\qquad
\lambda=1+2\log(1+W).
\end{equation}
If \(\Ric_g\geq-Kg\) on \(\Omega\), then
\begin{equation}\label{eq:subsolution}
\divg(A\nabla W)\geq-2K\lambda W.
\end{equation}
\end{lemma}

\begin{proof}
Equation \eqref{eq:nondiv} gives
\begin{equation}\label{eq:pde-h}
\Delta u=-\ip{\nabla u}{\nabla h}.
\end{equation}
By the Bochner formula,
\begin{align*}
\Delta h
&=|\nabla^2u|^2+\Ric(\nabla u,\nabla u)
  +\ip{\nabla u}{\nabla\Delta u}\\
&=|\nabla^2u|^2+\Ric(\nabla u,\nabla u)
  -|\nabla h|^2-\nabla^2h(\nabla u,\nabla u).
\end{align*}
Here we used \(\nabla_{\nabla u}\nabla u=\nabla h\).
Consequently,
\begin{equation}\label{eq:A-hessian-h}
A^{ij}\nabla_i\nabla_jh
=|\nabla^2u|^2+\Ric(\nabla u,\nabla u)-|\nabla h|^2.
\end{equation}

Set \(B^{ij}=e^h(g^{ij}+u^iu^j)\).
Since \(\nabla_i(e^hu^i)=0\), covariant differentiation gives
\begin{align*}
\nabla_iB^{ij}
&=e^hh^j+\nabla_i(e^hu^i)u^j
  +e^hu^i\nabla_i u^j\\
&=2e^hh^j.
\end{align*}
Note that \(\nabla W=e^h\nabla h\), so \(A\nabla W=B\nabla h\).
Together with \eqref{eq:A-hessian-h}, this yields
\begin{align*}
\divg(A\nabla W)
&=B^{ij}\nabla_i\nabla_jh+(\nabla_iB^{ij})h_j\\
&=e^h (\Delta h+\nabla^2h(\nabla u,\nabla u)+2|\nabla h|^2)\\
&=e^h\left(
|\nabla^2u|^2+\Ric(\nabla u,\nabla u)+|\nabla h|^2
\right).
\end{align*}
This proves \eqref{eq:bochner}.

For every tangent vector \(X\),
\[
\ip{AX}{X}=|X|^2+\ip{\nabla u}{X}^2
\leq(1+|\nabla u|^2)|X|^2.
\]
Since \(1+|\nabla u|^2=1+2\log(1+W)\),
inequality \eqref{eq:ellipticity} follows.

Under \(\Ric_g\geq-Kg\), equation \eqref{eq:bochner} implies
\[
\divg(A\nabla W)
\geq-2K(1+W)\log(1+W).
\]
For \(s\geq0\), we have
\begin{align*}
(1+s)\log(1+s)
&\leq s\log(1+s)+s\\
&\leq s\bigl(1+2\log(1+s)\bigr).
\end{align*}
Taking \(s=W\) proves \eqref{eq:subsolution}.
\end{proof}

We now apply the Bochner identity to obtain a rigidity result.
For related splitting results involving functions with parallel
gradient, see \cite{Mendonca}.

\begin{corollary}\label{cor:gradient-rigidity}
Suppose that \(M\) is complete and connected, and \(\Ric_g\geq0\).
Let \(u\in C^3(M)\) solve \eqref{eq:eh}.
If \(|\nabla u|\) attains its global maximum at some \(p\in M\), then
\[
\nabla^2u=0,\qquad \Ric_g(\nabla u,\nabla u)=0
\quad\text{on }M.
\]
In particular, either \(u\) is constant, or \(M\) is isometric to
\((\R\times N,dt^2+g_N)\), where \(N\) is complete and connected
and \(\Ric_{g_N}\geq0\).
Under this isometry,
\[
u(t,y)=at+b,\qquad a=|\nabla u|(p)>0,\qquad b=u(p).
\]
\end{corollary}

\begin{proof}
We use \(h\), \(W\), and \(A\) defined in \eqref{eq:hAW}.
By \eqref{eq:bochner} and \(\Ric_g\geq0\), we have
\[
\divg(A\nabla W)\geq0.
\]
Set \(a=|\nabla u|(p)\). Since \(|\nabla u|\leq a\), we have
\[
|X|^2\leq\ip{AX}{X}\leq(1+a^2)|X|^2.
\]
Moreover, \(W\) attains its global maximum at \(p\).
Applying the strong maximum principle, we see that \(W\) is
constant on \(M\). Using \eqref{eq:bochner} again, we get
\[
|\nabla^2u|^2+\Ric_g(\nabla u,\nabla u)+|\nabla h|^2=0.
\]
Each term is nonnegative, so \(\nabla^2u=0\) and
\(\Ric_g(\nabla u,\nabla u)=0\).
If \(a=0\), then \(u\) is constant.

Now assume \(a>0\). Set \(V=\nabla u/a\).
Then \(V\) is a parallel unit vector field.
Its integral curves are unit-speed geodesics, so its flow
\(\phi_t\) is defined for all \(t\in\R\).
Since \(V\) is parallel, each \(\phi_t\) is an isometry.
Moreover, 
\[
u(\phi_t(x))=u(x)+at.
\]
Set \(b=u(p)\) and \(N=u^{-1}(b)\), with the induced metric \(g_N\).
Since \(du\neq0\), the level set \(N\) is a smooth hypersurface.
Define
\[
F:\R\times N\longrightarrow M,\qquad F(t,y)=\phi_t(y).
\]
This map has the inverse
\[
F^{-1}(x)=\bigl(s(x),\phi_{-s(x)}(x)\bigr),
\qquad s(x)=\frac{u(x)-b}{a}.
\]
Thus \(F\) is a global diffeomorphism.
Since \(TN=V^\perp\) and the flow preserves both \(g\) and \(V\),
we obtain
\[
F^*g=dt^2+g_N,\qquad u(F(t,y))=at+b.
\]
It follows from this product decomposition that \(N\) is complete
and connected. The product formula for the Ricci tensor also gives
\(\Ric_{g_N}\geq0\).
\end{proof}

\begin{remark}\label{rem:gradient-rigidity}
If, in addition, \(\Ric_g\) is positive definite at one point,
then \(\Ric_g(\nabla u,\nabla u)=0\) implies that \(u\) is constant.
\end{remark}

\section{A mean-value estimate}\label{sec:iteration}

In this section, we prove a mean-value estimate to be used with
\eqref{eq:subsolution}. We keep the dependence on the upper
ellipticity bound explicit. A key tool is the local Sobolev
inequality of Saloff-Coste, which we recall first.

Suppose that $(M^n,g)$ is complete and $\Ric_g\geq-Kg$,
where $K\geq0$.
Let $B=B_r(x)$ and $V=\Vol_g(B)$.
For $n>2$, set $q=n/(n-2)$.
By \cite[Theorem~3.1]{SC},
\[
\left(\int_B|\varphi|^{2q}\dd\right)^{1/q}
\leq e^{C_n(1+\sqrt K\,r)}V^{-2/n}
\left(
r^2\int_B|\nabla\varphi|^2\dd+
\int_B|\varphi|^2\dd
\right),
\quad \varphi\in C_c^\infty(B).
\]
Here $C_n$ depends only on $n$.
For $n=2$, the same estimate holds with $n$ replaced by any
fixed number $n'>2$.
In this case, $q=n'/(n'-2)$, the volume factor is $V^{-2/n'}$,
and the constant depends on $n'$.

We now write this inequality using normalized integrals.
Set
\begin{equation}\label{eq:exponents}
d=n+2,\qquad \chi=\frac{d}{d-2}.
\end{equation}
Take $r=2\rho$ and assume that $K\rho^2\leq1$.
When $n=2$, choose a fixed number $n'=4$, so that $q=2=\chi$.
When $n>2$, keep $q=n/(n-2)$. Dividing the Sobolev inequality by $V^{1/q}$ converts all
integrals into normalized integrals.
Indeed, $1/q+2/n=1$ when $n>2$, and
$1/q+2/n'=1$ when $n=2$.
In both cases, $\chi\leq q$.
Hence H\"older's inequality with respect to normalized volume gives
\[
\left(\fint_B|\varphi|^{2\chi}\dd\right)^{1/\chi}
\leq
\left(\fint_B|\varphi|^{2q}\dd\right)^{1/q}.
\]
Since $r=2\rho$ and $\sqrt K\,r\leq2$, we obtain
\begin{equation}\label{eq:sobolev}
\left(\fint_B|\varphi|^{2\chi}\dd\right)^{1/\chi}
\leq C_n\left(
\rho^2\fint_B|\nabla\varphi|^2\dd+
\fint_B|\varphi|^2\dd
\right),
\qquad \varphi\in C_c^\infty(B).
\end{equation}
By density, \eqref{eq:sobolev} also holds for compactly
supported Lipschitz functions.
The zero-order term allows us to use the same inequality
when $M$ is compact and $B=M$.

\begin{lemma}\label{lem:mean-value}
Suppose that \(\Ric_g\geq-Kg\) with \(K\geq0\), and \(K\rho^2\leq1\).
Let \(B=B_{2\rho}(x)\) and \(d=n+2\).
Suppose that \(A\) is a \(C^1\) symmetric endomorphism field on \(B\)
and
\[
|X|^2\leq\ip{AX}{X}\leq\lambda|X|^2
\]
for a measurable function \(\lambda\geq1\).
Let \(w\in C^2(B)\) be nonnegative and satisfy
\begin{equation}\label{eq:abstract-subsolution}
\divg(A\nabla w)\geq-2K\lambda w
\end{equation}
in the weak sense.
If
\[
\Lambda=\left(\fint_B\lambda^d\dd\right)^{1/d}<\infty,
\qquad m=\fint_Bw\dd<\infty,
\]
then
\begin{equation}\label{eq:mean-value}
\sup_{B_\rho(x)}w\leq C_n\Lambda^d m.
\end{equation}
\end{lemma}

\begin{proof}
First replace \(w\) by \(w+\varepsilon\) with \(\varepsilon>0\).
The same differential inequality holds, since
\[
\divg(A\nabla(w+\varepsilon))
\geq-2K\lambda w\geq-2K\lambda(w+\varepsilon).
\]
The coefficient field \(A\) and the function \(\lambda\) remain fixed.
We prove the estimate for positive \(w\), with constants independent
of \(\varepsilon\), and then let \(\varepsilon\downarrow0\).

All norms in this proof are taken with respect to the normalized measure
\[
d\nu=\frac{dV_g}{\Vol(B)}.
\]
For \(U\subset B\), write
\[
\|f\|_{a,U}=\left(\int_U|f|^a\,d\nu\right)^{1/a}.
\]

We fix \(\rho\leq s<t<2\rho\) and choose a Lipschitz
cutoff function \(\eta\) such that
\[
0\leq\eta\leq1,\qquad
\eta=1\ \text{on }B_s(x),\qquad
\operatorname{supp}\eta\subset B_t(x),\qquad
|\nabla\eta|\leq\frac2{t-s}.
\]
For \(\beta\geq2\), testing \eqref{eq:abstract-subsolution}
with \(\eta^2w^{\beta-1}\), we obtain
\begin{equation}\label{eq:test}
\begin{split}
(\beta-1)\int_B\eta^2w^{\beta-2}
 \ip{A\nabla w}{\nabla w}\,d\nu
&\leq
2\int_B\eta w^{\beta-1}
 |\ip{A\nabla w}{\nabla\eta}|\,d\nu
\\
&\quad+2K\int_B\lambda\eta^2w^\beta\,d\nu.
\end{split}
\end{equation}
Cauchy's inequality in the \(A\)-inner product yields
\begin{equation}\label{eq:caccioppoli}
\begin{split}
\int_B\eta^2w^{\beta-2}
 \ip{A\nabla w}{\nabla w}\,d\nu
&\leq\frac4{(\beta-1)^2}
 \int_Bw^\beta\ip{A\nabla\eta}{\nabla\eta}\,d\nu
\\
&\quad+\frac{4K}{\beta-1}
\int_B\lambda\eta^2w^\beta\,d\nu.
\end{split}
\end{equation}
Since \(A\geq I\) and \(\beta/(\beta-1)\leq2\), we also have
\begin{equation}\label{eq:product-energy}
\begin{split}
\int_B|\nabla(\eta w^{\beta/2})|^2\,d\nu&\leq 2\int_Bw^\beta|\nabla\eta|^2\,d\nu+\frac{\beta^2}{2}
 \int_B\eta^2w^{\beta-2}
 |\nabla w|^2\,d\nu\\
&\leq2\int_Bw^\beta|\nabla\eta|^2\,d\nu+\frac{\beta^2}{2}
 \int_B\eta^2w^{\beta-2}
 \ip{A\nabla w}{\nabla w}\,d\nu\\
 &\leq10\int_B\lambda w^\beta|\nabla\eta|^2\,d\nu+4K\beta\int_B\lambda\eta^2w^\beta\,d\nu.
\end{split}
\end{equation}
Then we apply \eqref{eq:sobolev} to \(\eta w^{\beta/2}\).
Using \(\lambda\geq1\) and \eqref{eq:product-energy}, we obtain
\begin{align*}
\|w\|_{\beta\chi,B_s}^{\beta}
&\leq C_n\left[
 \frac{\rho^2}{(t-s)^2}+K\rho^2\beta+1
 \right]\int_{B_t}\lambda w^\beta\,d\nu\\
&\leq C_n\beta
 \left(\frac{\rho}{t-s}\right)^2
 \int_{B_t}\lambda w^\beta\,d\nu.
\end{align*}
In the last step, we use \(K\rho^2\leq1\) and \(t-s\leq\rho\).
Set
\[
q=\frac d{d-1},\qquad
\sigma=\frac\chi q=\frac{d-1}{d-2}>1.
\]
By H\"older's inequality,
\begin{equation}\label{eq:iteration-step}
\|w\|_{\beta\chi,B_s}
\leq
\left[
 C_n\beta\Lambda
 \left(\frac{\rho}{t-s}\right)^2
\right]^{1/\beta}
\|w\|_{\beta q,B_t}.
\end{equation}

For fixed \(s<t\), take
\[
t_j=s+(t-s)2^{-j},\qquad
\beta_j=2\sigma^j,\qquad p_j=q\beta_j,
\qquad j=0,1,\ldots.
\]
Then \(\beta_j\chi=p_{j+1}\).
Applying \eqref{eq:iteration-step} successively from
\(B_{t_j}(x)\) to \(B_{t_{j+1}}(x)\), we get
\[\|w\|_{p_{j+1},B_{t_{j+1}}}
\leq
\left[
 C_n\beta_j\Lambda
 \left(\frac{2^{j+1}\rho}{t-s}\right)^2
\right]^{1/\beta_j}
\|w\|_{p_{j},B_{t_{j}}}.\]
Note that
\[
\sum_{j=0}^{\infty}\frac1{\beta_j}=\frac{d-1}{2},
\qquad
\sum_{j=0}^{\infty}\frac{j+1}{\beta_j}<\infty,
\qquad
\sum_{j=0}^{\infty}\frac{\log\beta_j}{\beta_j}<\infty.
\]
Thus, with \(p_0=2d/(d-1)\), we obtain
\begin{equation}\label{eq:lp-to-linfty}
H(s):=\sup_{B_s(x)}w=\|w\|_{\infty,B_s}
\leq C_n\Lambda^{(d-1)/2}
 \left(\frac{\rho}{t-s}\right)^{d-1}
 \|w\|_{p_0,B_t}.
\end{equation}

Next we reduce the right-hand side to \(L^1\).
Since
\[
\|w\|_{p_0,B_t}
\leq H(t)^{1-1/p_0}m^{1/p_0},
\]
Young's inequality gives, for every fixed \(\theta\in(0,1)\),
\begin{equation}\label{eq:hole-filling}
H(s)\leq\theta H(t)+
C_{n,\theta}\Lambda^d m
\left(\frac{\rho}{t-s}\right)^{2d}
\end{equation}
where we use
\[
\frac{d-1}{2}p_0=d,\qquad (d-1)p_0=2d.
\]
Choose
\[
\theta=2^{-(2d+1)},\qquad
r_j=\frac{3\rho}{2}-\frac{\rho}{2^{j+1}}.
\]
Then \(r_0=\rho\), \(r_j\uparrow3\rho/2\), and
\(r_{j+1}-r_j=\rho2^{-(j+2)}\).
Iterating \eqref{eq:hole-filling}, we get
\[
H(\rho)\leq\theta^JH(r_J)
+C_n\Lambda^d m
 \sum_{j=0}^{J-1}\theta^j2^{2d(j+2)}.
\]
The series converges because \(\theta2^{2d}=1/2\).
Also, \(w\) is bounded on \(\overline{B_{3\rho/2}(x)}\).
Letting \(J\to\infty\) proves \eqref{eq:mean-value}
for positive \(w\).
Finally, letting \(\varepsilon\downarrow0\), we finish the proof.
\end{proof}

\begin{proposition}\label{prop:energy-mv}
Suppose that \(\Ric_g\geq-Kg\) and \(K\rho^2\leq1\).
Let \(u\in C^3(B_{4\rho}(x))\) solve \eqref{eq:eh}.
Set
\[
W=e^{|\nabla u|^2/2}-1,\qquad
m=\fint_{B_{2\rho}(x)}W\dd.
\]
Then
\begin{equation}\label{eq:energy-mv}
\sup_{B_\rho(x)}W\leq C_nm(1+m).
\end{equation}
\end{proposition}

\begin{proof}
The integral defining \(m\) is finite because
\(\overline{B_{2\rho}(x)}\) is a compact subset of \(B_{4\rho}(x)\).
By Lemma~\ref{lem:bochner}, \(W\) satisfies
\[
\divg(A\nabla W)\geq-2K\lambda W,\qquad
I\leq A\leq\lambda I,\qquad
\lambda=1+2\log(1+W).
\]
For the fixed exponent \(d=n+2\), there is a constant \(c_d\) such that
\[
[1+2\log(1+s)]^d\leq c_d(1+s),\qquad s\geq0.
\]
It follows that
\[
\Lambda^d=\fint_{B_{2\rho}(x)}\lambda^d\dd
\leq c_d(1+m).
\]
Applying Lemma~\ref{lem:mean-value} to \(W\) yields \eqref{eq:energy-mv}.
\end{proof}

\section{The oscillation gradient estimate}\label{sec:gradient}

We first prove the estimate on balls satisfying \(K\rho^2\leq1\).

\begin{proposition}\label{prop:small-ball}
Suppose \(\Ric_g\geq-Kg\) and \(K\rho^2\leq1\).
If \(u\in C^3(B_{4\rho}(x))\cap L^\infty(B_{4\rho}(x))\)
solves \eqref{eq:eh}, then
\begin{equation}\label{eq:small-ball}
\sup_{B_\rho(x)}|\nabla u|
\leq\frac{C_n}{\rho}\osc_{B_{4\rho}(x)}u.
\end{equation}
\end{proposition}

\begin{proof}
Set
\[
S=\osc_{B_{4\rho}(x)}u,\qquad
z=\frac{S^2}{32\rho^2},\qquad
m=\fint_{B_{2\rho}(x)}W\dd.
\]
Applying Lemma~\ref{lem:energy}, we obtain
\[
m\leq a_n(e^z-1),\qquad a_n\geq1.
\]
Together with Proposition~\ref{prop:energy-mv}, this yields
\begin{align*}
\sup_{B_\rho(x)}W
&\leq C_na_n(e^z-1)\bigl[1+a_n(e^z-1)\bigr]\\
&\leq D_n(e^{2z}-1)
\end{align*}
for some constant \(D_n\geq1\).
In the last inequality, we use
\[
1+a_n(e^z-1)\leq a_ne^z,\qquad
e^z(e^z-1)\leq e^{2z}-1.
\]
Since \(|\nabla u|^2=2\log(1+W)\), we conclude that
\[
\begin{split}
\sup_{B_\rho(x)}|\nabla u|^2&=2\log(1+\sup_{B_\rho(x)} W)\\
&\leq 2\log(1+D_n(e^{2z}-1)) \leq 4D_nz
=\frac{D_n}{8}\frac{S^2}{\rho^2}.
\end{split}
\]
Taking square roots proves \eqref{eq:small-ball}.
\end{proof}

\begin{proof}[Proof of Theorem~\ref{thm:gradient}]
Fix \(x\in B_R(p)\).
If \(K>0\), set
\[
\rho=\min\left\{\frac R2,\frac1{\sqrt K}\right\}.
\]
If \(K=0\), set \(\rho=R/2\).
In both cases,
\[
K\rho^2\leq1,\qquad
B_{4\rho}(x)\subset B_{3R}(p)\subset B_{4R}(p).
\]
Proposition~\ref{prop:small-ball} therefore gives
\[
|\nabla u(x)|
\leq\frac{C_n}{\rho}\osc_{B_{4\rho}(x)}u
\leq\frac{C_n}{\rho}\osc_{B_{4R}(p)}u.
\]
Moreover,
\[
\frac1\rho\leq\frac2R+\sqrt K.
\]
Taking the supremum over \(B_R(p)\) and adjusting the constant \(C_n\)
proves \eqref{eq:gradient}.
\end{proof}

\section{The one-sided gradient estimate}\label{sec:one-sided}

In this section, we prove Theorem \ref{thm:one-sided}. We first bound a positive solution by its value at a single point.
Set
\begin{equation}\label{eq:flux-function}
a(s)=s e^{s^2/2},\qquad s\geq0,
\qquad k=\max\{n-1,2\}.
\end{equation}
The function \(a\) is strictly increasing from \([0,\infty)\)
onto \([0,\infty)\), so its inverse \(a^{-1}\) is well defined.
Also,
\begin{equation}\label{eq:Ik}
I_k:=\int_0^\infty a(s)^{-1/k}\,ds
=\int_0^\infty s^{-1/k}e^{-s^2/(2k)}\,ds<\infty.
\end{equation}
For every \(Q>0\), Tonelli's
theorem gives
\begin{equation}\label{eq:inverse-flux-integral}
\begin{split}
\int_0^\infty a^{-1}(Qt^{-k})\,dt
&=\int_0^\infty\int_0^\infty
  \mathbf{1}_{\{a(s)<Qt^{-k}\}}\,ds\,dt\\
&=Q^{1/k}\int_0^\infty a(s)^{-1/k}\,ds
=I_kQ^{1/k}.
\end{split}
\end{equation}

\begin{lemma}\label{lem:positive-bound}
Suppose that \(\Ric_g\geq-Kg\) with \(K\geq0\).
Let \(v\in C^3(B_{3L}(z))\) be a positive exponentially harmonic
function. Set \(b=\sqrt{(n-1)K}\). Then
\begin{equation}\label{eq:positive-bound}
\sup_{B_L(z)}v
\leq 2I_k e^{2bL/k}L
       \left[a\!\left(\frac{v(z)}L\right)\right]^{1/k}.
\end{equation}
\end{lemma}

\begin{proof}
Fix \(y\in B_L(z)\), and write \(r=d_g(y,\cdot)\).
The closed ball \(\overline{B_{2L}(y)}\) is a compact subset of
\(B_{3L}(z)\).
For \(Q>0\), define
\begin{equation}\label{eq:radial-profile}
f_Q(r)=\int_r^{2L}a^{-1}(Qt^{-k}e^{-bt})\,dt,
\qquad 0\leq r\leq2L.
\end{equation}
By \eqref{eq:inverse-flux-integral},
\begin{equation}\label{eq:profile-height}
0<f_Q(0)\leq I_kQ^{1/k}.
\end{equation}
Thus \(h_Q=f_Q\circ r\) is continuous at \(y\).
It is locally Lipschitz away from \(y\), and vanishes at
\(r=2L\).

Let \(J(\xi)=e^{|\xi|^2/2}\xi\). Then the exponentially harmonic function $v$ satisfies 
\begin{equation}\label{eq:div-exp}
    \divg (J(\nabla v))=0.
\end{equation}
Since \(|\nabla r|=1\) almost everywhere away from \(y\),
\begin{equation}\label{eq:barrier-flux}
J(\nabla h_Q)=-Qr^{-k}e^{-br}\nabla r.
\end{equation}
For \(K>0\), set \(\kappa=\sqrt{K/(n-1)}\).
The standard Laplacian comparison bound is
\[
\Delta r\leq(n-1)\kappa\coth(\kappa r).
\]
Using \(\coth t\leq1+t^{-1}\) and \(b=\sqrt{K(n-1)}\), we obtain
\begin{equation}\label{eq:laplacian-comparison}
\Delta r\leq\frac{n-1}{r}+\sqrt{K(n-1)}=\frac{n-1}{r}+b
\end{equation}
away from \(y\), in the weak sense.
The case \(K=0\) follows directly.

Set
\[
G_Q(r)=\int_r^{2L}Qt^{-k}e^{-bt}\,dt.
\]
Then \(J(\nabla h_Q)=\nabla(G_Q\circ r)\) almost everywhere.
Moreover,
\[
G_Q'(t)=-Qt^{-k}e^{-bt},\qquad
G_Q''(t)=Qt^{-k}e^{-bt}\left(\frac{k}{t}+b\right).
\]
Since \(G_Q'\leq0\), we apply the weak Laplacian comparison
\cite[Theorem~3.6]{Wei} to obtain
\begin{equation}\label{eq:barrier-subsolution}
\begin{split}
\divg J(\nabla h_Q)
&=\Delta(G_Q\circ r)\\
&\geq G_Q''(r)
 +\left(\frac{n-1}{r}+b\right)G_Q'(r)\\
&=(k-n+1)Qr^{-k-1}e^{-br}\geq0
\end{split}
\end{equation}
on \(B_{2L}(y)\setminus\{y\}\), in the distributional sense.

The value $f_Q(0)=\int_0^{2L}a^{-1}(Qt^{-k}e^{-bt})\,dt$ is continuous and
strictly increasing in \(Q\). It tends to zero as \(Q\downarrow0\)
by \eqref{eq:profile-height} and to infinity as \(Q\to\infty\),
since
\begin{equation}\label{eq:profile-lower}
f_Q(0)\geq L a^{-1}\!\left((2L)^{-k}e^{-2bL}Q\right).
\end{equation}
Hence for any fixed \(0<\alpha<1\), we can choose \(Q>0\) so that
\begin{equation}\label{eq:height-choice}
f_Q(0)=\alpha v(y).
\end{equation}
We claim that \(h_Q=f_Q\circ r\leq v\) on \(B_{2L}(y)\).
By \eqref{eq:height-choice}, we have \(h_Q<v\) near \(y\).
Also, \(h_Q=0<v\) on the outer boundary, if nonempty.
Thus \(\psi=(h_Q-v)_+\) has compact support away from \(y\)
and the outer boundary. Testing \eqref{eq:barrier-subsolution} and the equation \eqref{eq:div-exp} for \(v\) with \(\psi\), we have
\[
\int_{B_{2L}(y)}
\ip{J(\nabla h_Q)-J(\nabla v)}{\nabla\psi}\dd\leq0.
\]
For tangent vectors \(\xi,\zeta\) at the same point,
\[
\ip{J(\xi)-J(\zeta)}{\xi-\zeta}\geq|\xi-\zeta|^2,
\]
since \(DJ(\xi)=e^{|\xi|^2/2}(I+\xi\otimes\xi)\geq I\).
We obtain
\[
\int_{B_{2L}(y)}|\nabla\psi|^2\dd=0.
\]
Since \(\psi\) vanishes near \(y\), we conclude that
\(\psi\equiv 0\). Thus \(h_Q\leq v\) on \(B_{2L}(y)\).

Since \(d_g(y,z)<L\), evaluating at \(z\) yields
\[
v(z)\geq h_Q(z)
\geq\int_L^{2L}a^{-1}(Qt^{-k}e^{-bt})\,dt
\geq L a^{-1}\!\left((2L)^{-k}e^{-2bL}Q\right).
\]
Applying the function \(a\) gives
\[
Q\leq(2L)^k e^{2bL}a\!\left(\frac{v(z)}L\right).
\]
Together with \eqref{eq:profile-height} and
\eqref{eq:height-choice}, this implies
\[
\alpha v(y)\leq2I_k e^{2bL/k}L
       \left[a\!\left(\frac{v(z)}L\right)\right]^{1/k}.
\]
Letting \(\alpha \uparrow 1\), and then taking the supremum over \(y \in B_L(z)\), proves \eqref{eq:positive-bound}.
\end{proof}

\begin{proof}[Proof of Theorem~\ref{thm:one-sided}]
Suppose first that \(u\geq c\) on \(B_{4R}(p)\).
Fix \(x\in B_R(p)\), and set
\[
\ell=\begin{cases}
R,&K=0,\\
\min\{R,K^{-1/2}\},&K>0.
\end{cases}
\]
Then \(K\ell^2\leq1\) and
\(B_{3\ell}(x)\subset B_{4R}(p)\).
For \(\varepsilon>0\), the function
\(v=u-c+\varepsilon\) is positive and exponentially harmonic. According to Lemma~\ref{lem:positive-bound}, we have
\[
\sup_{B_\ell(x)}v
\leq C_n\ell
       \left[a\!\left(\frac{v(x)}\ell\right)\right]^{1/k},
\]
since
\(b\ell=\sqrt{K(n-1)}\ell\leq\sqrt{n-1}\). Applying Proposition~\ref{prop:small-ball} with \(\rho=\ell/4\) and using \(v>0\), we obtain
\begin{align*}
|\nabla u(x)|
&=|\nabla v(x)|
\leq\frac{C_n}{\ell}\osc_{B_\ell(x)}v\\
&\leq\frac{C_n}{\ell}\sup_{B_\ell(x)}v
\leq C_n\left[a\!\left(\frac{v(x)}\ell\right)\right]^{1/k}
=C_n\Psi_k\!\left(\frac{v(x)}\ell\right).
\end{align*}
Now use
\[
\ell^{-1}=\max\{R^{-1},\sqrt K\}
\leq \frac{1}{R}+\sqrt K.
\]
Letting \(\varepsilon\downarrow0\) yields the required gradient estimate \eqref{eq:one-sided-gradient}.
If \(u\leq c\), we simply apply the result to \(-u\).
\end{proof}

\section{Liouville theorems}\label{sec:liouville}

As a consequence of the oscillation estimate, we first consider the sublinear exponentially harmonic functions.
This also gives a proof of the bounded Liouville theorem.

\begin{proposition}\label{prop:oscillation-liouville}
Suppose that \(M\) is complete and connected, and \(\Ric_g\geq0\).
Let \(u\in C^3(M)\) be exponentially harmonic.
If, for some \(p\in M\),
\begin{equation}\label{eq:sublinear}
\osc_{B_R(p)}u=o(R)\qquad\text{as }R\to\infty,
\end{equation}
then \(u\) is constant. In particular, every bounded exponentially
harmonic function on \(M\) is constant.
\end{proposition}

\begin{proof}
Fix \(x\in M\), and take \(R>d_g(p,x)\).
Completeness implies that \(u\) is bounded on each ball
\(B_{4R}(p)\).
With \(K=0\), Theorem~\ref{thm:gradient} gives
\[
|\nabla u(x)|\leq\frac{C_n}{R}\osc_{B_{4R}(p)}u.
\]
By \eqref{eq:sublinear}, the right-hand side tends to zero as
\(R\to\infty\). Thus \(\nabla u=0\), and connectedness implies
that \(u\) is constant. If \(u\) is bounded, then
\[
\osc_{B_R(p)}u\leq2\|u\|_{L^\infty(M)}=o(R),
\]
so the same conclusion holds.
\end{proof}

\begin{remark}
The bounded case of Proposition~\ref{prop:oscillation-liouville}
is a Riemannian special case of Shen's Liouville theorem
on Finsler manifolds \cite[Theorem~1.1]{Shen}.
Our proof is different.
\end{remark}

Next, we use the one-sided estimate to control one-sided growth, namely Theorem~\ref{thm:liouville}.

\begin{proof}[Proof of Theorem~\ref{thm:liouville}]
Assume first that
\(u(p)-\inf_{B_R(p)}u=o(R)\).
Fix \(x\in M\). For \(R>d_g(p,x)\), set
\[
c_R=\inf_{B_{4R}(p)}u.
\]
This number is finite because \(u\) is continuous and
\(\overline{B_{4R}(p)}\) is compact.
With \(K=0\), Theorem~\ref{thm:one-sided} gives
\[
|\nabla u(x)|\leq C_n\Psi_k\!\left(\frac{u(x)-c_R}{R}\right).
\]
We also have
\[
\frac{u(x)-c_R}{R}
=\frac{u(x)-u(p)}R+\frac{u(p)-c_R}R\longrightarrow0
\]
as \(R\to\infty\). Since \(\Psi_k\) is continuous and \(\Psi_k(0)=0\), we get
\(\nabla u(x)=0\). The point \(x\) was arbitrary, so \(u\) is
constant. The upper growth condition follows by applying the same
argument to \(-u\).
\end{proof}

\begin{proof}[Proof of Corollary~\ref{cor:bounded}]
If \(u\geq c\) on \(M\), then
\[
0\leq u(p)-\inf_{B_R(p)}u\leq u(p)-c=o(R).
\]
Theorem~\ref{thm:liouville} therefore applies.
If \(u\) is bounded above, apply the result to \(-u\).
\end{proof}

\begin{remark}\label{rem:linear-growth}
The sublinear assumption in Theorem \ref{thm:liouville} is sharp. In fact, on the Euclidean space, every affine function
\(u(x)=a\cdot x+b\) solves \eqref{eq:eh}.
For \(a\neq0\),
\[
|\nabla u|=|a|,\qquad
u(0)-\inf_{B_R(0)}u=\sup_{B_R(0)}u-u(0)=|a|R.
\]
Also, \(\osc_{B_R(0)}u=2|a|R\).
Thus linear growth on either side allows nonconstant solutions.
\end{remark}

\end{document}